\documentclass[final,leqno,letterpaper]{etna}

\setbibdata{1}{xx}{46}{2017} 

\hypersetup{%
    pdftitle={Near-optimal convergence of the full orthogonalization method},
    pdfauthor={Tyler Chen, G\'erard Meurant},
    pdfkeywords={FOM,GMRES}
    }

\title{Near-optimal convergence of the full orthogonalization method}

\author{Tyler Chen\footnotemark[2] \and G\'erard Meurant\footnotemark[3]}

\shorttitle{Near-optimal convergence of the full orthogonalization method} 
\shortauthor{T.~CHEN AND G.~MEURANT}

\usepackage{amsmath}
\usepackage{graphicx}
\usepackage{hyperref}
\hypersetup{
    colorlinks,
    citecolor=red
}
\usepackage[nameinlink,capitalize,noabbrev]{cleveref}

\crefformat{equation}{#2(#1)#3}
\crefmultiformat{equation}{#2(#1)#3}%
{ and~#2(#1)#3}{, #2(#1)#3}{ and~#2(#1)#3}
\crefrangeformat{equation}{(#3#1#4) to (#5#2#6)}

\crefname{section}{Section}{Sections}
\crefname{subsection}{Subsection}{Subsections}

\renewcommand{\d}{\mathrm{d}}
\renewcommand{\vec}{\mathbf}
\newcommand{\T}{\mathsf{T}}

\makeatletter
\def\@Stamppage{}
\def\@Stamptitle{}
\def\@Stamppagetitle{}
\makeatother

\begin{document}

\maketitle

\renewcommand{\thefootnote}{\fnsymbol{footnote}}

\footnotetext[2]{New York University, New York, USA.}
\footnotetext[3]{Paris, France.}

\begin{abstract}
We establish a near-optimality guarantee for the full orthogonalization method (FOM), showing that the \emph{overall} convergence of FOM is nearly as good as GMRES. 
In particular, we prove that at every iteration $k$, there exists an iteration $j\leq k$ for which the FOM residual norm at iteration $j$ is no more than $\sqrt{k+1}$ times larger than the GMRES residual norm at iteration $k$.
This bound is sharp, and it has implications for algorithms for approximating the action of a matrix function on a vector.\end{abstract}

\begin{keywords}
Full Orthogonalization Method, GMRES
\end{keywords}

\begin{AMS}
65F10
\end{AMS}

\section{Introduction}
The \emph{full orthogonalization method} (FOM) \cite{saad_81} and the \emph{generalized minimal residual method} (GMRES) \cite{saad_schultz_86} are two Krylov subspace methods used for solving a real or complex non-symmetric\footnote{If $\vec{A}$ is symmetric GMRES is mathematically equivalent to MINRES \cite{paige_saunders_75}, and if $\vec{A}$ is symmetric positive definite FOM is mathematically equivalent to conjugate gradient \cite{hestenes_stiefel_52}. While this is relevant for efficient implementation, it does not impact the exact arithmetic theory in this note.} linear system of equations
\begin{equation}
    \label{eqn:Axb}
\vec{A}\vec{x} = \vec{b}.
\end{equation}
Assuming an initial guess $\vec{x}_0 = \vec{0}$, both FOM and GMRES produce iterates from the Krylov subspace
\begin{equation}
\label{eqn:krylov_subspace}
    \mathcal{K}_k(\vec{A},\vec{b}) 
    := \operatorname{span}\{\vec{b}, \vec{A}\vec{b}, \ldots, \vec{A}^{k-1}\vec{b}\}.
\end{equation}
but according to slightly different formulas.
The FOM iterate is of particular interest, because it is closely related to the Arnoldi method for matrix-function approximation for approximating $f(\vec{A})\vec{b}$, the action of a matrix function on a vector \cite{druskin_knizhnerman_89,gallopoulos_saad_92,higham_08}.
We expand on this connection in \cref{sec:fAb}

Denote by $\vec{Q}_k = [\vec{q}_1, \ldots, \vec{q}_k]$ the orthonormal basis for the Krylov subspace $\mathcal{K}_k(\vec{A},\vec{b})$ produced by the Arnoldi algorithm \cite{arnoldi_51}.
Define also the $(k+1)\times k$ upper-Hessenberg matrix $\vec{H}_{k+1,k}$ of coefficients produced by the Arnoldi algorithm, and recall the Arnoldi recurrence relation
\begin{equation}
    \vec{A}\vec{Q}_k = \vec{Q}_{k+1}\vec{H}_{k+1,k}.    
\end{equation}
The FOM and GMRES iterates (with zero initial guess $\vec{x}_0^{\mathrm{F}}  = \vec{x}_0^{\mathrm{G}}  = \vec{0}$) are respectively defined as 
\begin{equation}\label{eqn:iterates}
    \vec{x}_k^{\mathrm{F}} 
    := \|\vec{b}\|_2 \vec{Q}_k (\vec{H}_k)^{-1} \vec{e}_1
    ,\qquad
    \vec{x}_k^{\mathrm{G}} 
    := \|\vec{b}\|_2 \vec{Q}_k (\vec{H}_{k+1,k})^\dagger \vec{e}_1,
\end{equation}
where $\vec{H}_k$ is $\vec{H}_{k+1,k}$ with the last row deleted and $\dagger$ indicates the pseudoinverse, and $\vec{e}_1 = [1,0,\ldots, 0]^\T$ is the first canonical basis vector \cite{meurant_tebbens_20}. 

Define the FOM and GMRES residual vectors
\begin{equation}   
    \vec{r}_k^{\mathrm{F}} := \vec{b} - \vec{A} \vec{x}_k^{\mathrm{F}}
    ,\qquad
    \vec{r}_k^{\mathrm{G}} := \vec{b} - \vec{A} \vec{x}_k^{\mathrm{G}}.
\end{equation}
The norms of these residuals can be used as measure of how well the iterates $\vec{x}_k^{\mathrm{F}}$ and $\vec{x}_k^{\mathrm{G}}$ solve the linear system \cref{eqn:Axb}, and understanding their relationship is the aim of this paper.

\begin{figure}
    \centering
    \includegraphics[scale=.48]{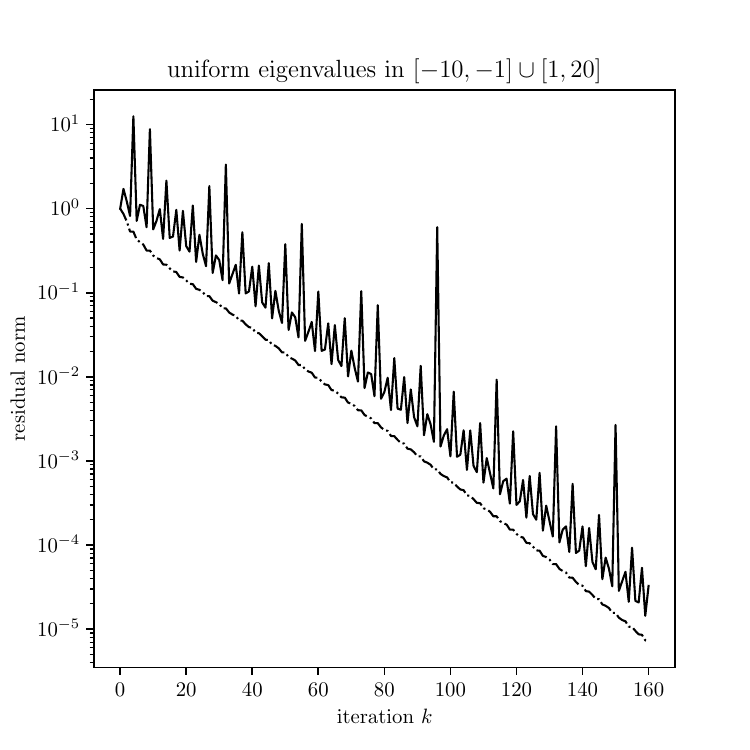}
    \includegraphics[scale=.48]{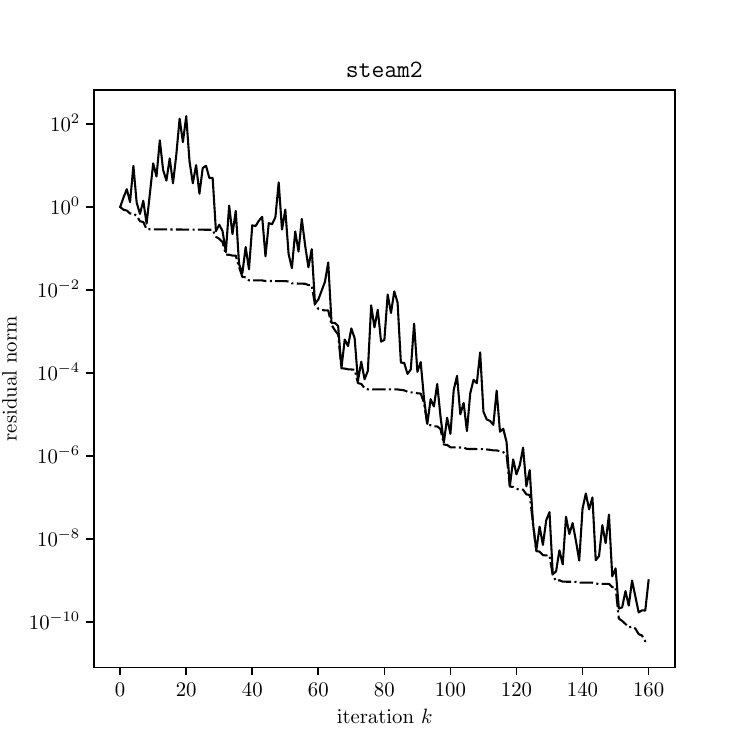}
    \caption{Residual norms for FOM $\|\vec{r}_k^{\mathrm{F}}\|_2$ (solid) and GMRES $\|\vec{r}_k^{\mathrm{G}}\|_2$ (dash-dot) for a symmetric matrix with eigenvalues in $[-10,-1]\cup[1,20]$ and the \texttt{steam2} matrix.
    All curves are normalized by $\|\vec{b}\|_2$.
    While the residual norms for FOM jump up and down, they exhibit a generally downward trend which mirrors the convergence of the GMRES residual norms.}
    \label{fig:motivation}
\end{figure}

It is well-known that the GMRES iterates satisfy a residual optimality guarantee:
\begin{equation}\label{eqn:GMRES_opt}
    \vec{x}_k^{\mathrm{G}} 
    = \operatornamewithlimits{argmin}_{\vec{x}\in\mathcal{K}_k(\vec{A},\vec{b})} \| \vec{b} - \vec{A} \vec{x} \|_{2}.
\end{equation}
Hence, the GMRES residual norms are non-increasing and are optimal among Krylov subspace methods.
This optimality guarantee leads to well-known results on the rate of convergence of the residual norms in terms of quantities such as the condition number of $\vec{A}$ \cite{saad_03}.
On the other hand, the FOM residual norms often appear oscillatory, with large jumps.
In fact, it is easy to construct examples for which $\|\vec{r}_k^{\mathrm{F}}\|_2/\|\vec{r}_0^{\mathrm{F}}\|_2$ can be arbitrarily large \cite{meurant_tebbens_20}!

Typical examples of the convergence behavior of FOM and GMRES are illustrated in \cref{fig:motivation}.
Here we consider a symmetric matrix with 500 eigenvalues equally spaced in $[-10,-1]\cup[1,20]$ and the \texttt{steam2} matrix from the Matrix Market \cite{boisvert_pozo_remington_barrett_dongarra_97}. 
In both cases we choose $\vec{b}$ as the all ones vector.
We observe that while GMRES exhibits well-behaved non-increasing residual norms, the residual norms for FOM are highly oscillatory. In fact, at some iterations, the FOM residual norms are orders of magnitude larger than those of GMRES. 
Remarkably, however, FOM exhibits an ``overall downward trend'' mirroring the convergence of GMRES.

\section{Convergence bounds}

\begin{figure}
    \centering
    \includegraphics[scale=.48
]{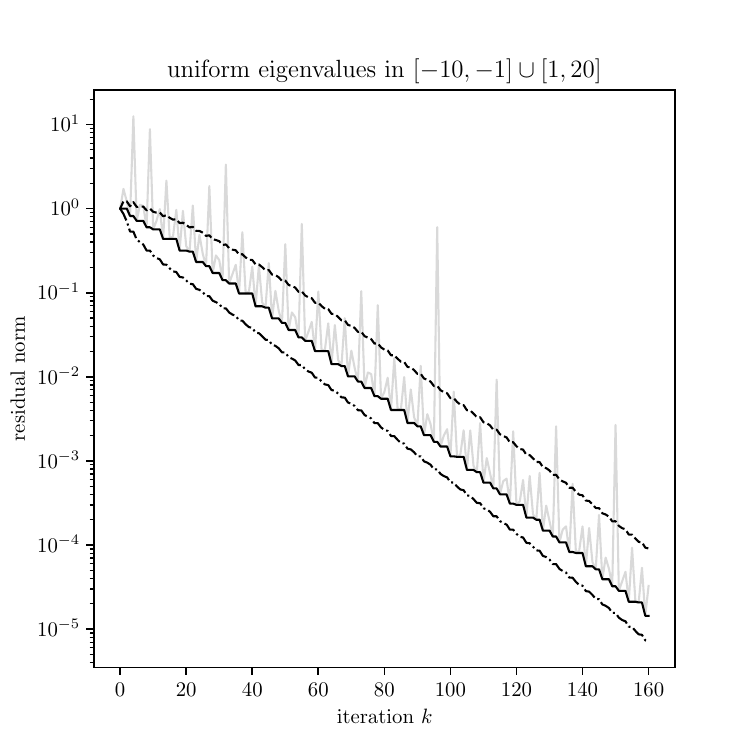}
    \hfil
    \includegraphics[scale=.48
]{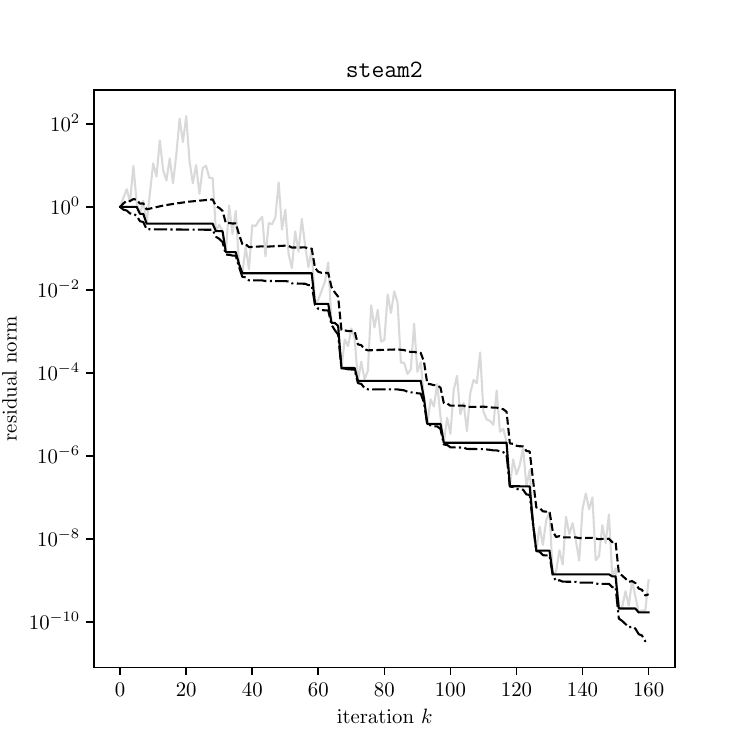}
    \caption{Best FOM residual so far $\min_{j\leq k} \|\vec{r}_j^{\mathrm{F}}\|_2$ (solid), bound of \cref{thm:main} (dashed), and residual norms for FOM $\|\vec{r}_k^{\mathrm{F}}\|_2$ (solid grey) and GMRES $\|\vec{r}_k^{\mathrm{G}}\|_2$ (dash-dot).
    All curves are normalized by $\|\vec{b}\|_2$.
    While FOM exhibits oscillatory convergence, \cref{thm:main} ensures that the best FOM residual norm seen so far matches closely with the GMRES residual norm.}
    \label{fig:result}
\end{figure}
The purpose of this note is to establish the following bound for the FOM residual norms in terms of the optimal GMRES residual norms:
\begin{theorem}\label{thm:main}
For every $k\geq 1$,
\begin{equation*}
    \min_{0\leq j\leq k} \|\vec{r}_{j}^{\mathrm{F}}\|_2
    \leq \sqrt{k+1} \cdot \|\vec{r}_k^{\mathrm{G}}\|_2.
\end{equation*}
\end{theorem}
The theorem asserts that, while the FOM residual norm at any given iteration $k$ can be arbitrarily large, the overall convergence of FOM is at most $\sqrt{k+1}$ worse than that of GMRES. 
In light of the optimality of the GMRES iterates \cref{eqn:GMRES_opt}, this implies the FOM iterates are, in a certain sense, near-optimal.
While an immediate consequence of existing work, we have not been able to find this bound in the literature; see \cref{sec:past} for a discussion on past work.
As we discuss in \cref{sec:discussion}, \cref{thm:main} has implications for commonly used Krylov subspace methods for matrix functions.

In many practical situations, the GMRES residual norm is exponential in $k$, in which case the factor $\sqrt{k+1}$ is comparatively unimportant.
In \cref{fig:result}, we show the examples from \cref{fig:motivation} along with the bound from \cref{thm:main}.
To emphasize the ``overall convergence'' of FOM, we also display the smallest residual norm of FOM seen up to a given iteration $k$: $\min_{j\leq k} \|\vec{r}_j^{\mathrm{F}}\|_2$.
As expected, this quantity tracks very closely the convergence of GMRES.

Our proof makes use of a characterization of the GMRES residual norms in terms of the FOM residual norms.
Define $\vartheta_{1}=1$ and for $k\geq 1$ define
\begin{equation}
\vartheta_{ k+1}=-\frac{1}{h_{k+1, k}} \sum_{j=1}^k \vartheta_{ j} h_{j, k},
\end{equation}
where $h_{i,j}$ is the $(i,j)$ entry of $\vec{H}_{k+1,k}$.
The FOM and GMRES residuals norms are related in the following sense:
\begin{theorem}[Theorem 3.12 in \cite{meurant_tebbens_20}]\label{thm:MT3.12}
For every $k\geq 0$,
\begin{equation*}
    \frac{\|\vec{r}_k^{\mathrm{F}}\|_2}{\|\vec{r}_0^{\mathrm{F}}\|_2}
    = \frac{1}{|\vartheta_{k+1}|},
    \qquad
    \frac{\|\vec{r}_k^{\mathrm{G}}\|_2}{\|\vec{r}_0^{\mathrm{G}}\|_2}
    = {\Bigg(\sum_{j=1}^{k+1} |\vartheta_{j}|^2 \Bigg)^{-1/2}}.    
\end{equation*}
\end{theorem}

The proof of \cref{thm:main} is an immediate consequence of \cref{thm:MT3.12}.

{\em Proof of \cref{thm:main}.}
Note that $\vec{H}_k$ is singular if and only if $\vartheta_{k+1} = 0$.
In this case, the FOM iterate is undefined and we can define the FOM residual norm to be infinite.
Therefore, defining $1/\infty = 0$, we have that
\begin{equation}
\frac{\|\vec{r}_k^{\mathrm{G}}\|_2}{\|\vec{r}_0^{\mathrm{G}}\|_2}
= \Bigg( \sum_{j=0}^{k} \frac{\|\vec{r}_0^{\mathrm{F}}\|_2^2} {\|\vec{r}_{j}^{\mathrm{F}}\|_2^2} \Bigg)^{-1/2}.
\end{equation}
Canceling $\|\vec{r}_0^{\mathbf{F}}\|_2 = \|\vec{r}_0^{\mathbf{G}}\|_2$ we therefore find
\begin{equation}\label{eqn:rkG}
\|\vec{r}_k^{\mathrm{G}}\|_2
= \Bigg( \sum_{j=0}^{k} \frac{1}{\|\vec{r}_{j}^{\mathrm{F}}\|_2^2} \Bigg)^{-1/2}.
\end{equation}
Thus, bounding each term in the sum by the maximum,
\begin{equation}
\|\vec{r}_k^{\mathrm{G}}\|_2
\geq \Bigg( (k+1) \max_{0\leq j\leq k} \frac{1}{\|\vec{r}_j^{\mathrm{F}}\|_2^2} \Bigg)^{-1/2}
= \frac{1}{\sqrt{k+1}} \cdot \min _{0\leq j \leq k} \|\vec{r}_j^{\mathrm{F}}\|_2,
\end{equation}
which proves the result.
\endproofhere

One may wonder whether the pre-factor $\sqrt{k+1}$ in \cref{thm:main} is necessary. 
We show that no better value is possible:
\begin{theorem}
    For every $k\geq 1$ there exists a matrix $\vec{A}$ and vector $\vec{b}$ for which 
    \begin{equation*}
    \min_{j\leq k} \|\vec{r}_j^{\mathrm{F}} \|_2 
    = 
    \sqrt{k+1} \cdot \|\vec{r}_k^{\mathrm{G}}\|_2.
    \end{equation*}
\end{theorem}

\begin{proof}
It is well-known that any sequence of residual norms is possible for FOM; that is, given any sequence of non-zero positive numbers $f_0,f_1, \ldots, f_k$, there exists a matrix $\vec{A}$ and vector $\vec{b}$ (of dimension $n=k+1$) for which $\|\vec{r}_j^{\mathrm{F}}\|_2 = f_j$ \cite[Theorem 3.15]{meurant_tebbens_20}.
In particular, there exists a matrix $\vec{A}$ and vector $\vec{b}$ for which $\|\vec{r}_j^{\mathrm{F}}\|_2^2 = 1$ for $j=0,1, \ldots, n$.
In this case, \cref{eqn:rkG} implies 
\[
\|\vec{r}_k^{\mathrm{G}}\|_2^2
= \frac{1}{\sum_{j=0}^{k} 1/\|\vec{r}_{j}^{\mathrm{F}}\|_2^2} 
= \frac{1}{k+1},
\]
which, since $\min_{0\leq j \leq k} \|\vec{r}_j^{\mathrm{F}}\|_2=1$, gives the result.
\end{proof}

\section{Discussion}\label{sec:discussion}

\subsection{Past work}\label{sec:past}

Many works have studied the relation between convergence of FOM and GMRES \cite[etc.]{brown_91,cullum_greenbaum_96,greenbaum_97,meurant_11a,meurant_11,fong_saunders_12}, particularly with the goal of understanding how to efficiently estimate the residual/error norms in practice. 

Perhaps the most well-known \cite{brown_91,cullum_greenbaum_96} relation between the FOM and GMRES residual norms is 
\begin{equation}\label{eqn:residuals}
    \|\vec{r}_k^{\mathrm{F}}\|_2
    = \frac{\| \vec{r}_{k}^{\mathrm{G}} \|_2}{\sqrt{1- \left( \| \vec{r}_{k}^{\mathrm{G}} \|_2 / \| \vec{r}_{k-1}^{\mathrm{G}} \|_2 \right)^2}},
\end{equation}
which can be obtained by rearranging \cref{eqn:rkG}.
Informally, \cref{eqn:residuals} says that at iterations where GMRES makes good progress (i.e. for which $\| \vec{r}_{k}^{\mathrm{G}} \|_2 \ll \| \vec{r}_{k-1}^{\mathrm{G}} \|_2$) the FOM residual norm is close to the GMRES residual norm, and at iterations where GMRES stagnates (i.e. for which $\| \vec{r}_{k}^{\mathrm{G}} \|_2 \approx \| \vec{r}_{k-1}^{\mathrm{G}} \|_2$) the FOM residual norm is very large.
However, it does not provide any information about at which, if any, iterations GMRES makes good progress.
As such, it is unclear from \cref{eqn:residuals} that FOM's overall convergence must track that of GMRES.

There are at least two bounds reminiscent of \cref{thm:main} in that they aim to understand the ``overall convergence'' of FOM, rather than the convergence at every iteration \cite{greenbaum_druskin_knizhnerman_99,chen_greenbaum_musco_musco_22}.
While \cite{greenbaum_druskin_knizhnerman_99} uses entirely different techniques, \cite{chen_greenbaum_musco_musco_22} uses \cref{eqn:residuals} to argue that if GMRES is making good progress, there must be some iteration for which $\| \vec{r}_{j}^{\mathrm{G}} \|_2$ is substantially smaller than $\| \vec{r}_{j-1}^{\mathrm{G}} \|_2$ and hence for which $\| \vec{r}_{k}^{\mathrm{G}} \|_2 \approx \| \vec{r}_{k-1}^{\mathrm{G}} \|_2$. 
Both bounds are weaker than \cref{thm:main} in that they (i) hold only for symmetric matrices and (ii) are not in terms of the GMRES residual norm, but rather the best approximation of zero on sets of the form $[a,b]\cup[c,b]$ by a polynomial taking value one at the origin, where $b<0<c$.

\subsection{Error norm}

We may wish to obtain a solution with small \emph{error} norm $\|\vec{A}^{-1}\vec{b} - \vec{x}\|_2$. 
It is typically not clear whether GMRES or FOM is preferable, and in many cases the error norms of FOM are smaller than those of GMRES.
In \cref{fig:error} we have illustrated the error norms for the examples from \cref{fig:motivation}, and observe that the FOM errors are slightly better.

Using that $\vec{A}^{-1}\vec{b} - \vec{x} = \vec{A}^{-1}(\vec{b} - \vec{A}\vec{x})$, \cref{thm:main} implies
\begin{equation}
\label{eqn:errornorm}
    \min_{j\leq k} \|\vec{A}^{-1}\vec{b} - \vec{x}_{j}^{\mathrm{F}}\|_2
    \leq  \sqrt{k+1} \cdot \kappa(\vec{A})\cdot  \|\vec{A}^{-1}\vec{b} - \vec{x}_k^{\mathrm{G}}\|_2.
\end{equation}
Therefore, so long as $\vec{A}$ is reasonably well-conditioned, the FOM error norms can, on the whole, never be significantly worse than those of GMRES.
The error norms of FOM and GMRES can be estimated a posteriori to determine when to stop \cite{meurant_11a}.

\begin{figure}
    \centering
    \includegraphics[scale=.48]{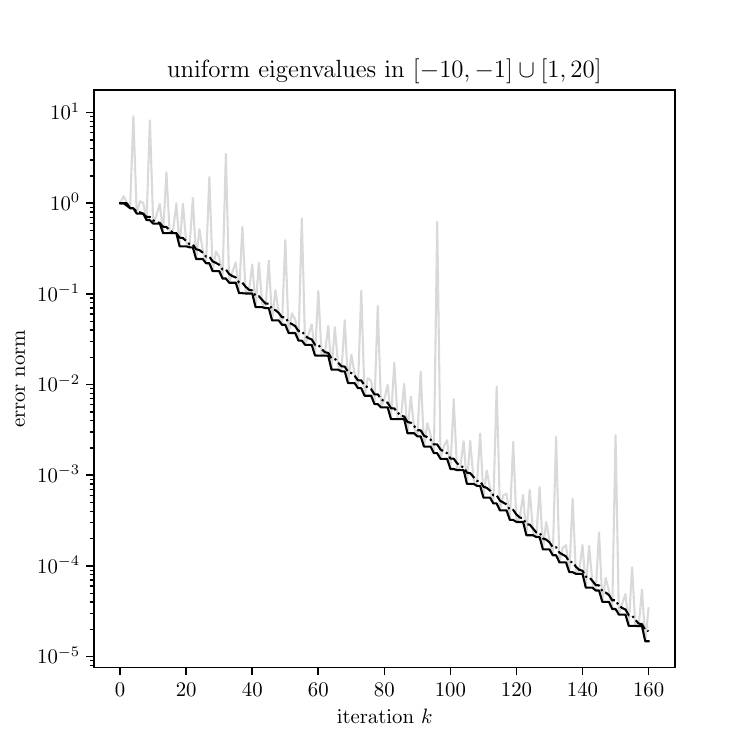}
    \includegraphics[scale=.48]{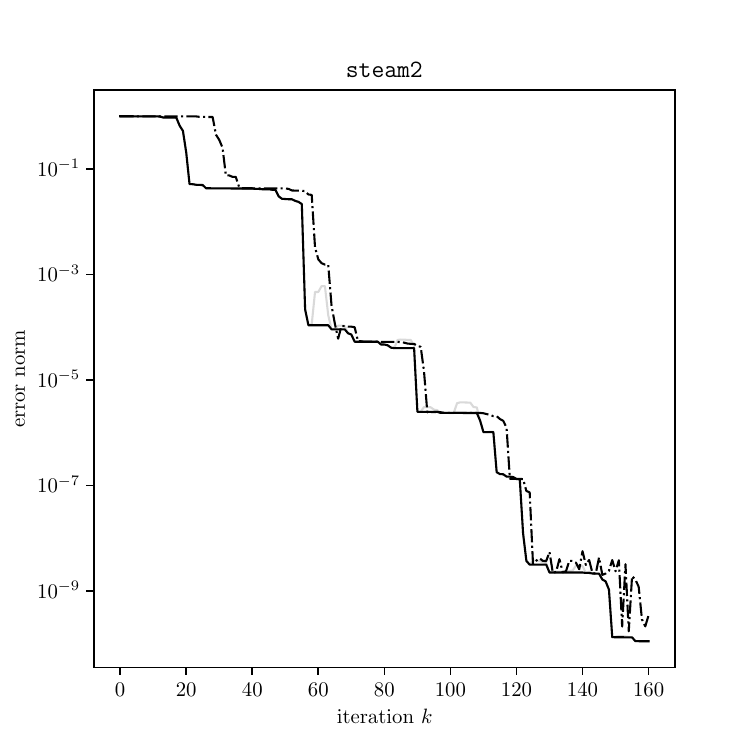}
    \caption{Best FOM error so far $\min_{j\leq k} \|\vec{A}^{-1}\vec{b} - \vec{x}_k^{\mathrm{F}}\|_2$ (solid), error for FOM $\|\vec{A}^{-1}\vec{b} - \vec{x}_k^{\mathrm{F}}\|_2$ (solid grey) and GMRES $\|\vec{A}^{-1}\vec{b} - \vec{x}_k^{\mathrm{G}}\|_2$ (dash-dot).
    All curves are normalized by $\|\vec{A}^{-1}\vec{b}\|_2$.
    Note that the FOM convergence is slightly better than GMRES, indicating FOM may be preferable to GMRES in some situations.}
    \label{fig:error}
\end{figure}

\subsection{Connection to matrix functions}\label{sec:fAb}
For convenience we will assume $\vec{A}$ is diagonalizable.
FOM is closely related to the Arnoldi method for matrix function approximation (Arnoldi-FA) \cite{druskin_knizhnerman_89,gallopoulos_saad_92,higham_08} which approximates $f(\vec{A})\vec{b}$ with iterates $\vec{Q}_k f(\vec{H}_k) \vec{e}_1$.
This is arguably the most widely used Krylov subspace method for approximating $f(\vec{A})\vec{b}$.

Commonly $f(x)$ can be expressed in the form
\begin{equation}
    \label{eqn:fintegral}
    f(x) = \int_\Gamma g(x)(x-z)^{-1} \d{x},
\end{equation}
for some weight function $g(x)$ and contour $\Gamma$. 
Common examples of functions which can be expressed in the form \cref{eqn:fintegral} include analytic functions (exponential, indicator function for a region in the complex plane) for which $\Gamma$ is a contour in the complex plane, Stieltjes functions (inverse square root, logarithm) for which $\Gamma$ is a subset of the real line, and rational functions with simple poles, which are a special case of Stieltjes functions; see for instance \cite{frommer_simoncini_08}.
Arnoldi-FA is widely used to compute the action of the corresponding matrix functions on a vector in applications throughout the computational sciences, including quantum chromodynamics, differential equations, machine learning, etc. \cite{frommer_simoncini_08,higham_08}.


When $f(x)$ can be written as \cref{eqn:fintegral} at the eigenvalues of $\vec{A}$ and $\vec{H}_k$, the Arnoldi-FA error is expressed as
\begin{equation}
f(\vec{A})\vec{b} - \vec{Q}_k f(\vec{H}_k) \vec{e}_1
= \int_\Gamma \Big[ 
(\vec{A} - z\vec{I})^{-1} \vec{b}
- \vec{Q}(\vec{H}_k - z\vec{I})^{-1} \vec{e}_1 \Big] \, \d\mu(z).
\end{equation}
Note that $\vec{Q}(\vec{H}_k - z\vec{I})^{-1}\vec{e}_1$ is the FOM approximation to the linear system $(\vec{A} - z\vec{I})\vec{x} = \vec{b}$. 
Hence, understanding the behavior of FOM is important to understanding the behavior of Arnoldi-FA.

Empirically, while Arnoldi-FA can sometimes have oscillatory convergence, it tends to follow a general downward trend approximately matching the best possible approximation from Krylov subspace \cite{amsel_chen_greenbaum_musco_musco_23}. 
In fact, we are unaware of any examples for which this is not the case.
Near-optimality guarantees have been proved in the symmetric case for the matrix exponential \cite{druskin_greenbaum_knizhnerman_98} and certain  rational functions \cite{amsel_chen_greenbaum_musco_musco_23}. 
\Cref{thm:main} is a near-optimality guarantee for Arnoldi-FA with $f(x) = 1/x$.

\section{Conclusion}

We have shown that the FOM residual norms are nearly as good as the GMRES residual norms in an overall sense. 
This provides theoretical justification for the use of FOM on linear systems of equations, as well as insight into the remarkable convergence of the Arnoldi method for matrix function approximation.

\section*{Acknowledgements}

We thank the editor and anonymous referees for their helpful feedback which improved the paper.

\bibliography{refs}
\bibliographystyle{amsplain}

\end{document}